\documentclass[11pt]{amsart}
\usepackage[english]{babel}
\usepackage[utf8]{inputenc}
\usepackage{amsmath}
\usepackage{amssymb}
\usepackage{amsfonts}
\usepackage{amsthm}
\usepackage{mathrsfs} 
\usepackage[all]{xy}
\usepackage[pdftex]{graphicx}
\usepackage{color} 
\usepackage{cite}
\usepackage{url}
\usepackage{indent first}
\usepackage[labelfont=bf,labelsep=period,justification=raggedright]{caption}
\usepackage[english]{babel}
\usepackage[utf8]{inputenc}
\usepackage{hyperref}
\usepackage[colorinlistoftodos]{todonotes}
\usepackage{tkz-fct} 
\usepackage[margin=1in]{geometry} 
\usepackage{tikz}

\DeclareMathOperator{\Prob}{Prob}

\DeclareMathOperator{\MC}{\textsf{Max-Cut}}
\DeclareMathOperator{\MLC}{\textsf{Max-$l$-Cut}}

\newcommand{\eps}{\varepsilon}

\newcommand{\EE}{\mathbb{E}}

\newcommand{\mcH}{\mathcal{H}}

\makeatletter
\newcommand\thankssymb[1]{\textsuperscript{\@fnsymbol{#1}}}
\makeatother

\theoremstyle{plain}
\newtheorem{thm}{Theorem}

\newtheorem{lemma}[thm]{Lemma}
\newtheorem{cor}[thm]{Corollary}
\newtheorem{conj}[thm]{Conjecture}
\newtheorem{prop}[thm]{Proposition}

\theoremstyle{definition}
\newtheorem{defn}[thm]{Definition}
\newtheorem{q}[thm]{Question}

\theoremstyle{remark}
\newtheorem{rem}[thm]{Remark}

\numberwithin{equation}{section}
\numberwithin{thm}{section}

\begin{document}

\title{Making an $H$-free graph $k$-colorable}
\author{Jacob Fox\thankssymb{1}}
\thanks{\thankssymb{1} Department of Mathematics, Stanford University, Stanford, CA 94305. Email: {\tt jacobfox@stanford.edu}. Research supported in part by a Packard Fellowship and by NSF Award DMS-1855635.}
\author{Zoe Himwich\thankssymb{2}}
\thanks{\thankssymb{2} Department of Mathematics, Columbia University, New York, NY. Email:
\href{mailto:himwich@math.columbia.edu} {\nolinkurl{himwich@math.columbia.edu}}.
}
\author{Nitya Mani\thankssymb{3}}
\thanks{\thankssymb{3} Department of Mathematics, Massachusetts Institute of Technology, Cambridge, MA. Email: \href{mailto:nmani@mit.edu} {\nolinkurl{nmani@mit.edu}}.
Research supported in part by a Hertz Fellowship, an MIT Presidential Fellowship, and the NSF GRFP}

\maketitle

\begin{abstract}
We study the following question: how few edges can we delete from any $H$-free graph on $n$ vertices in order to make the resulting graph $k$-colorable? It turns out that various classical problems in extremal graph theory are special cases of this question. For $H$ any fixed odd cycle, we determine the answer up to a constant factor when $n$ is sufficiently large. We also prove an upper bound when $H$ is a fixed clique that we conjecture is tight up to a constant factor, and prove upper bounds for more general families of graphs. We apply our results to get a new bound on the maximum cut of graphs with a forbidden odd cycle in terms of the number of edges.  
\end{abstract}

\section{Introduction}
All graphs we consider are finite, undirected and simple, unless otherwise specified. A graph is {\it $H$-free} if it does not contain $H$ as a subgraph. For a collection $\mathcal{H}$ of graphs, a graph is $\mathcal{H}$-free if it does not contain any graph in $\mathcal{H}$ as a subgraph. The {\it girth} of a graph is the length of the shortest cycle, and it is infinite if the graph is a forest. The \textit{chromatic number} $\chi(G)$ of a graph $G$ is the minimum number of colors needed to properly color the vertices of the graph so that no two adjacent vertices receive the same color. 

A famous result of Erd\H{o}s \cite{E59} states that there are graphs of arbitrarily large girth and chromatic number. While these graphs are locally sparse, they cannot be properly colored with few colors. We study here a slightly different local-global problem in graphs with a similar flavor: how resilient to being $k$-colorable can a graph be given a local constraint like a forbidden subgraph?

Precisely, for a graph $G$ and a positive integer $k$, how few edges, which we denote by $h(G,k)$, can we delete from $G$ in order to make the remaining subgraph $k$-colorable? For a graph $G$ and positive integers $n$ and $k$, let $h(n,k,H)$ be the maximum of $h(G,k)$ over all $n$-vertex graphs $G$ which are $H$-free, that is, which do not contain $H$ as a subgraph. We define $h(n,k,\mathcal{H})$ analogously for $\mcH$ a family of forbidden subgraphs. Determining or estimating $h(n,k,H)$ is a very challenging problem. Special cases of this problem include several famous problems in extremal graph theory. For example, the case $k=1$ is the classical Tur\'an problem on the maximum number of edges an $H$-free graph on $n$ vertices can have. 

A longstanding conjecture of Erd\H{o}s (he wrote in 1975 \cite{ER75} that it was already old) would solve the case where $H$ is a triangle and $k=2$. This conjecture states that every triangle-free graph on $n$ vertices can be made bipartite by deleting at most $n^2/25$ edges. If true, this conjectured bound is the best possible. This can be seen by considering a balanced blow-up of a cycle on five vertices. While there are many papers on this problem, the best known upper bound \cite{EGS92} is a little better than $n^2/18$. Solving another conjecture of Erd\H{o}s, Sudakov \cite{Su07} showed that any $K_4$-free graph on $n$ vertices can be made bipartite by removing at most $n^2/9$ edges. That is, $h(n,2, K_4) \le n^2/9$. This bound is tight, which can be seen by considering a balanced blow-up of a triangle. He deduced as a corollary that, if $H$ is a fixed graph with $\chi(H)=4$, then $h(n,2, H) \leq (1+o(1))n^2/9$. Sudakov further conjectured for $r>4$ that the balanced complete $(r-1)$-partite graph on $n$ vertices is the furthest from being bipartite over all $K_r$-free graphs, which would determine $h(n,2, K_r)$ for $r>3$. It was recently announced \cite{Lid18} that Hu, Lidick\' y, Martins-Lopez, Norin, and Volec verified the case $r=6$ of Sudakov's conjecture and further proved the corresponding upper bound for all $H$ of chromatic number $6$. 

The other extreme is determining the minimum $k$ for which $h(n,k,H)$ is zero. This is the same as determining the maximum possible chromatic number that an $H$-free graph on $n$ vertices can have. This problem is very old (see  \cite{E59}), and closely related to estimating the Ramsey number $r(H,K_s)$.  If $r(H,K_s) > n$ and $k \leq n/s$, then there is an $H$-free graph $G$ on $n$ vertices which does not contain an independent set of order $s$. In particular, in any $k$-coloring of the vertices, one of the color classes has at least $n/k \geq s$ vertices and must contain an edge, implying $h(n,k,H)>0$. On the other hand, if $H$ is connected, $k \geq (n/s) \cdot \log_2 (2n)$, and $n > r(H,K_s)$, then by greedily picking out largest independent sets, one can properly $k$-color any $H$-free graph on $n$ vertices. This bound on the number of colors can be deduced from the fact that the minimum possible independence number of an $H$-free graph on $n$ vertices is a monotonically increasing and subadditive function of $n$. 

In this paper, we will be primarily interested in the intermediate case, when $|H| \ll k \ll n$. Throughout the article, we use the notation $\Theta_{x}$, $\Omega_{x}$, $O_{x}$, and $o_{x}$ to indicate that the implicit constant factors may depend on $x$. 

Note that for any graph $G$ on $n$ vertices we have the simple bound $h(G, k) \le h(K_n,k) \leq {n \choose 2}/k \leq n^2/(2k)$ by considering a random $k$-partition of $V(G)$. Thus, we will be primarily interested in understanding how much of an improvement we can give over this straightforward bound for $H$-free graphs.

The following theorem gives an upper bound on $h(n, k, H)$ when $H$ is a clique on $r$ vertices.
 
\begin{thm}\label{t:clique} 
For each integer $r \geq 3$ there is $c_r$ such that for all positive integers $n, k$ we have $$h(n, k, K_r) \le c_r \frac{n^2}{k^{(r-1)/(r-2)}}.$$
\end{thm}

We conjecture that Theorem \ref{t:clique} is sharp up to the constant factor $c_r$ for $n$ sufficiently large in terms of $k$, and prove this for $r=3$. 

For fixed odd cycles and $n$ sufficiently large in $k$, we determine this function up to a constant factor. We first show an upper bound on $h(G, k)$ for graphs $G$ of large odd girth.

\begin{thm}\label{t:cyclegirth} For each positive integer $r$ there is $c_r$ such that the following holds.
Let $\mcH_r = \{C_3, C_5, \ldots C_{2r+1}\}$ be the set of odd cycles of length at most $2r + 1$. Then,
$$h(n, k, \mcH_r) <  c_r\frac{n^2}{k^{r+1}}.$$
\end{thm} 

From this result on graphs of large odd girth, we can deduce a similar result for graphs with a single forbidden odd cycle. 

\begin{thm}\label{t:oddcycle} 
For positive integers $n\geq k \ge 1$, and $r \ge 1$, we have 
$h(n, k, C_{2r+1}) = h(n, k, \mcH_r) + O_r(n^{3/2})$, where $\mcH_r$ is the family of odd cycles of length at most $2r+1$. In particular,
$$h(n, k, C_{2r+1})  = O_r\left(\frac{n^2}{k^{r+1}}\right).$$
\end{thm}
In the other direction, we prove the following lower bound which shows that Theorem \ref{t:oddcycle} is tight up to the constant factor in $r$ for $n$ sufficiently large in terms of $k$. The proof leverages a construction of Alon and Kahale~\cite{AL98} of a family of pseudorandom graphs of large odd girth.
\begin{thm}\label{c:lb} For each positive integer $r$ there is $\alpha_r>0$ such that the following holds. For each positive integer $k$, for each sufficiently large positive integer $n$, 
there is a graph $G$ on $n$ vertices with odd girth larger than $2r+1$ and with $h(G, k) \ge \alpha_r n^2/k^{r+1}.$
\end{thm}
In the case $H$ is a triangle, we determine up to a constant factor how large $n$ has to be for the above to hold using a result about a semi-random variant of the triangle-free process~\cite{BOH99,GUO20}.

The \textit{wheel} $W_{\ell}$ is the graph on $\ell + 1$ vertices consisting of an $\ell$-cycle and an additional vertex adjacent to all of the vertices of the $\ell$-cycle. For 
even wheels (when $\ell$ is even), we prove that $h(n, k, W_{\ell})$ is asymptotically the same as $h(n, k, K_3)$. Combining the methods used in the proofs of Theorems \ref{t:clique} and \ref{t:cyclegirth}, we prove the following upper bound for odd wheels, which we conjecture is tight up to the constant factor which depends on the length of the wheel. 

\begin{thm}\label{t:wheel} 
For each positive integer $r$ there is $c_r$ such that if $n \ge k \ge 2$, then
$$h(n, k, W_{2r+1}) \le c_r \frac{n^2}{k^{2 - 1/(r+1)}}.$$ 
\end{thm}

More broadly, we use the graph removal lemma in Section~\ref{s:homgen} to prove the following result, which shows that if $H$ has a subgraph $H'$ for which $H$ has a homomorphism to $H'$, then 
$h(n, k, H)$ and $h(n,k,H')$ are close. 

\begin{thm}\label{c:homomorphism}
If a graph $H$ has a subgraph $H'$ such that there exists a homomorphism from $H$ to $H'$, then 
$$h(n, k, H') \le h(n, k, H) \le h(n, k, H')+o(n^2).$$
\end{thm}

The maximum cut of a graph $G$, denoted by $\MC(G)$, is the maximum number of a bipartite subgraph. This well-studied graph parameter to the main focus of this paper 
through the identity $\MC(G)=e(G)-h(G,2)$, where $e(G)$ is the number of edges of $G$. It is a simple exercise to show that every graph $G$ with $m$ edges has $\MC(G) \geq m/2$. Edwards \cite{Ed73,Ed75} proved that this bound can be improved to $$\MC(G) \geq \frac{m}{2}+\frac{-1+\sqrt{8m+1}}{8},$$
which is sharp if $m={k \choose 2}$ for some positive integer $k$, as shown by taking $G=K_k$. Further results for intermediate values of $m$ were established in 
\cite{AL96,AH98,BO02}. 

There has been a lot of research on improving the lower order term in the Edwards bound for graphs with a fixed forbidden subgraph. 
Alon, Krivelevich, and Sudakov \cite{AL03} showed that for $H$ fixed and $m$ significantly large, every $H$-free graph $G$ with $m$ edges satisfies $\MC(G) \geq \frac{m}{2}+m^{1/2 +\epsilon}$ for some $\epsilon=\epsilon(H)>0$, and they conjectured that $1/2$ in the exponent can be replaced by $3/4$. 

A case of particular interest is when $H$ is a cycle. Solving a problem of Erd\H{o}s, Alon \cite{AL96} proved that every triangle-free graph $G$ with $m$ edges satisfies $\MC(G) \geq \frac{m}{2}+cm^{4/5}$ for some positive constant $c$, and this is tight up to the constant factor $c$. More generally, Alon et al. \cite{AL05} conjectured that for all $k \geq 3$, every $C_k$-free graph $G$ with $m$ edges satisfies $\MC(G) \geq \frac{m}{2}+\Omega_k(m^{(k+1)/(k+2)})$. They verified their conjecture for $k$ even, and showed that the conjectured bound is tight for $k \in \{4,6,10\}$.\footnote{While \cite{ZEN16} attributes this conjecture to \cite{AL05}, what is actually conjectured in \cite{AL05} is that for even $k$ the conjectured bound is tight.} Alon, Bollob\'as, Krivelevich, and Sudakov \cite{AL03} observed that for odd $k$, a well-known construction of Alon \cite{AL94,AH98} gives a pseudorandom graph $G$ with odd-girth greater than $k$, $m$ edges, and $\MC(G) = \frac{m}{2}+O_k(m^{(k+1)/(k+2)})$. This construction shows that if the Alon-Krivelevich-Sudakov conjecture is true, then the bound it gives is best possible for odd $k$. Recently, Zeng and Hou \cite{ZEN16} proved that for fixed odd $k$, every $C_k$-free graph $G$ with $m$ edges satisfies $\MC(G) \geq \frac{m}{2}+m^{(k+1)/(k+3)+o(1)}$. Using Theorem \ref{t:oddcycle} and some additional tools, we prove the following result giving an improved bound. 

\begin{thm}\label{max-cut-odd-cycle-free} 
If $k \geq 3$ is odd and $G=(V,E)$ is a $C_k$-free graph with $m$ edges, then 
$$\MC(G) \ge \frac{m}{2} + \Omega_k(m^{(k+5)/(k+7)}).$$
\end{thm}

\noindent {\bf Organization.} We begin in Section~\ref{s:neighbor} by developing tools that allow us to conclude some incidental results such as Proposition~\ref{p:mantel}, a strengthening of Mantel's theorem, and also give a foundation to prove Theorem~\ref{t:clique} and Theorem~\ref{t:wheel}.
In Section~\ref{s:clique}, we give an upper bound on $h(n, k, K_r)$, the number of edges that must be removed from an arbitrary $K_r$-free graph on $n$ vertices to guarantee the resulting subgraph is $k$-partite.

Subsequently in Section~\ref{s:oddcycle}, we give upper bounds on $h(n, k, \mcH)$ for $\mcH = \{C_3, \ldots C_{2r+1}\}$ and use this bound on graphs of large odd girth to obtain an upper bound on $h(n, k, C_{2r+1})$. Using a generalization of Alon's construction of a family of pseudorandom graphs of large odd girth, we show that our bound is tight up to a constant factor depending on $r$. In Section~\ref{s:graphhom}, we leverage the above bounds to obtain associated bounds on $h(n, k, H)$ for more other forbidden subgraphs $H$.

We apply our results to the problem of bounding the $\textsf{Max-$k$-Cut}$ of a graph, the size of the largest $k$-partite subgraph of a graph, noticing that $h(G,k)=e(G)-\textsf{Max-$k$-Cut}(G)$. We first give some simple lemmas to translate between bounds on maximum $k$-cuts and maximum $l$-cuts for $l < k$ in Section~\ref{s:lcuts}. This enables us in Section~\ref{s:maxcut} to prove Theorem \ref{max-cut-odd-cycle-free} giving a new lower bound on $\MC(G)$ for graphs with a forbidden odd cycle. Finally, we conclude in Section~\ref{s:conclude} with some unresolved open questions.

\section{Cutting Graphs using Neighborhoods}
For a graph $G$ and vertex subset $U \subset V(G)$, let $G[U]$ denote the induced subgraph of $G$ with vertex set $U$. We let $e(G)$ denote the number of edges of $G$, and $e(U)=e(G[U])$ denote the number of edges with both vertices in $U$. For a vertex $v$ of $G$, the neighborhood $N(v)$ is the set of vertices of $G$ adjacent to $v$. The degree of $v$, which is $|N(v)|$, is denoted $d(v)$.

In this section, we study the following extremal problem in graph theory.
\begin{q}\label{q:neighbor}
Given a graph $G$, how many edges of $G$ can we cover by the union of $k$ neighborhoods of vertices of $G$? 
\end{q}
In understanding Question~\ref{q:neighbor}, we will build up a series of tools that will be useful in our subsequent analysis of $h(n, k, H)$ when $H$ is a clique or an odd wheel. The methods we describe below are also of independent interest. In Section~\ref{s:neighbor} we include a few applications of this analysis beyond our study of how far graphs are from $k$-colorable.

\subsection{Covering edges with the union of neighborhoods}\label{s:neighbor}

We tackle Question~\ref{q:neighbor}, denoting the relevant value $u(G,k)$, which is defined formally below. 

\begin{defn}
For a graph $G$ and positive integer $k$, let $u(G,k)$ be the maximum of $e\left(\bigcup_{i=1}^k N(v_i)\right)$ over all choices of vertices $v_1,\ldots,v_k$ of $G$.
\end{defn}

We would like to understand how few edges can we leave uncovered by the union of $k$ neighborhoods of vertices of a graph on $n$ vertices. 

\begin{defn}
Let $m(n,k)$ be the minimum of $e(G)-u(G,k)$ over all graphs $G$ on $n$ vertices. That is, $m(n,k)$ is the minimum $r$ such that, for every graph $G$ on $n$ vertices, there are $k$ vertices $v_1, \ldots , v_k$ such that at most $r$ edges are not contained in
the induced subgraph $G[N(v_1) \cup \cdots \cup N(v_k)]$ whose vertex set is the union of the neighborhoods of $v_1, \ldots v_k$.
\end{defn}

As an aside, we first observe that $m(n,1) = \lfloor \frac{n^2}{4}\rfloor$. This is a strengthening of Mantel's theorem, that every triangle-free graph on $n$ vertices has at most $\lfloor \frac{n^2}{4}\rfloor$ edges, as it is easy to see that $u(G,1)=0$ if and only if $G$ is triangle-free. 

\begin{prop}\label{p:mantel}
Every graph $G$ on $n$ vertices has a vertex whose neighborhood contains all but at most $\lfloor \frac{n^2}{4} \rfloor$ edges of $G$, and this bound is sharp. That is, $m(n,1)=\lfloor \frac{n^2}{4}\rfloor$. 
\end{prop}
\begin{proof}
The balanced complete bipartite graph on $n$ vertices realizes $m(n,1) \geq \lfloor \frac{n^2}{4} \rfloor$. 

If $m(n,1) > \lfloor \frac{n^2}{4} \rfloor$, then it would be realized by a graph $G$ on $n$ vertices and $m=\frac{n^2}{4}+t$ edges with $t$ positive. A result of Moon and Moser (c.f.~\cite{PI12}) states that any graph on $n$ vertices and $m$ edges has at least $$\frac{m(4m-n^2)}{3n}=\left(\frac{n^2}{4}+t\right)\frac{4t}{3n} =\frac{tn}{3}+\frac{4t^2}{3n}$$ triangles. Hence, a random vertex of $G$ is in expectation at least $$\frac{3}{n} \cdot \left(\frac{tn}{3}+\frac{4t^2}{3n}\right)>t$$ triangles, and hence there is a vertex $v$ of $G$ where $e(N(v)) > t$. The number of edges of $G$ not in the neighborhood of $v$ is an integer which is less than $m-t  = \frac{n^2}{4}$, and hence $m(n,1) \leq \lfloor \frac{n^2}{4}\rfloor$. 
\end{proof}

The following result yields a lower bound on $u(G,k)$ by considering a random choice of $k$ vertices. 

\begin{lemma}\label{newrandom}
If $G=(V,E)$ is a graph on $n$ vertices and $k$ is a positive integer, then
\begin{equation}\label{firstbda}e(G)-u(G,k) \le \sum_{u \in V} d(u)\left(1-\frac{d(u)}{n}\right)^k-\sum_{\{u,w\} \in E}\left(1-\frac{|N(u) \cup N(w)|}{n}\right)^k.\end{equation}
\end{lemma}
\begin{proof}
Pick $k$ vertices $v_1, ..., v_k \in V$ uniformly at random with repetition. Let $U=\bigcup_{i=1}^k N(v_i)$. An edge $(u,w)$ of $G$ is not in $G[U]$
if and only if $v_1,\ldots,v_k$ are in $V \setminus N(u)$ or $V \setminus N(w)$. By the inclusion-exclusion principle, the probability that not both $u$ and $w$ are in $U$ is \begin{equation}\label{firstbd}\left(1-\frac{d(u)}{n}\right)^k+\left(1-\frac{d(w)}{n}\right)^k-\left(1-\frac{|N(u) \cup N(w)|}{n}\right)^k.\end{equation}
Splitting up the sum and then summing the first two terms over vertices of $G$, we find that the expected value of $e(G)-e(U)$ is at most the right hand side of (\ref{firstbda}). Hence, there is a choice of $v_1,\ldots,v_k$ such that $e(G)-e(U)$ (and hence $e(G)-u(G,k)$) is at most the right hand side of (\ref{firstbda}).
\end{proof}

We have no idea what the exact or asymptotic value of $m(n,k)$ is for any fixed $k \geq 2$. We will prove in general (using Lemma~\ref{newrandom}) that $m(n,k) \leq \frac{n^2}{ek}$, which, for $k$ sufficiently large and $n$ sufficiently large in terms of $k$, is within $20\%$ of the lower bound that comes from considering an appropriate Erd\H{o}s-Renyi random graph $G(n,p)$ with $p=c/k$. To see this, pick $c>0$ to maximize $ce^{-c}-\frac{c}{2}e^{-2c}$. Note that a simple union bound shows that almost surely all of the linear-sized induced subgraphs of $G(n,p)$ have edge density $(1+o(1))p$. This implies that the union of the neighborhood of any $k$ vertices has size $(1+o(1))\left(1-(1-p)^k\right)n$ and the induced subgraph will have edge density $(1+o(1))p$. Thus $$m(n,k) \geq (1+o_k(1))\left(ce^{-c}-\frac{c}{2}e^{-2c}\right)\frac{n^2}{k},$$ for $n$ sufficiently large as a function of $k$, and with the $o_k(1)$ term tending to $0$ as $k \to \infty$.  

\begin{cor}\label{cornewrandom}
We have $m(n,k) \leq \frac{n^2}{ek}$. That is, for every graph $G$ on $n$ vertices, there are $k$ vertices of $G$ such that the induced subgraph on the union of the neighborhoods of these $k$ vertices contains all but at most $\frac{n^2}{ek}$ edges of $G$. 
\end{cor}
\begin{proof}
By Lemma \ref{newrandom}, for any graph $G$ on $n$ vertices, we have $e(G)-u(G,k)$ is at most
\begin{equation}\label{randombdb} \sum_{u \in V} d(u)\left(1-\frac{d(u)}{n}\right)^k,\end{equation}
since the second sum in (\ref{firstbda}) is non-negative. The function $f(x)=x(1-\frac{x}{n})^k$ has derivative $f'(x)=\left(1-\frac{(k+1)x}{n}\right)(1-\frac{x}{n})^{k-1}$, which is non-negative for $x \leq \frac{n}{k+1}$ and is non-positive if $\frac{n}{k+1}<x \leq n$. Thus, $f(x)$ for $x \leq n$ is maximized at $x=\frac{n}{k+1}$. 
Hence, for $x \leq n$ we have
$$f(x) \le f\left(\frac{n}{k+1} \right) = \frac{1}{k+1}\left(1-\frac{1}{k+1}\right)^k n=\frac{1}{k}\left(1-\frac{1}{k+1}\right)^{k+1} n \leq \frac{n}{ek}.$$ 
Applying this with $x=d(u)$ in (\ref{randombdb}) gives the desired inequality. 
\end{proof}

\begin{lemma}\label{l:partpart}
If a graph $G=(V,E)$ has disjoint vertex subsets $U_1,\ldots,U_r$, then there is a partition of $V=W_1 \sqcup \cdots \sqcup W_r$ with $U_i \subset W_i$ for $1 \leq i \leq r$ and the number $\sum_{i=1}^r e(W_i)-e(U_i)$ of edges that are in a $G[W_i]$ but not $G[U_i]$ is at most $\left(e(G)-e(U_1 \cup \cdots \cup U_r)\right)/r$. 
\end{lemma}
\begin{proof}
Let $X=U_1 \cup \cdots \cup U_r$. For each vertex $v \in V \backslash X$, randomly add $v$ to one of the $r$ sets $U_{i}$. This gives a random partition of $V$ into $r$ sets. Let $W_i$ denote the part which is a superset of $U_i$. Each edge with not both of its vertices in $X$ has a probability $1/r$ that both of its vertices end up in the same part of the random partition. Hence, by linearity of expectation, the expected number of edges of $G$ with not both its vertices in $X$ that end up in the same part of the random partition is $(e(G) - e(X))/r$. So there is such a partition with at most $(e(G) - e(X))/r$ edges with not both its vertices in $X$ and which lie in the same part. 
\end{proof}

We remark that the above probabilistic proof of Lemma \ref{l:partpart} can be made deterministic by greedily assigning the vertices to the part that it has the fewest edges to. 

\begin{lemma}\label{l:hpart}
If $G=(V,E)$ is a graph with disjoint vertex subsets $V_1,\ldots,V_t$ and $s$ is a positive integer, then 
$$h(G,st) \leq \frac{1}{st}\left(e(G)-e(V_1 \cup \cdots \cup V_{t})\right)+\sum_{i=1}^t h(G[V_i],s).$$
\end{lemma}
\begin{proof}
For each $1 \le i \le t$, there is a partition $V_i = U_{i1} \sqcup \cdots \sqcup U_{is}$ so that we can make each of the $s$ vertex subsets $U_{ij}$ independent sets by removing at most $h(G[V_i],s)$ total edges. By Lemma \ref{l:partpart}, we can grow the $st$ disjoint subsets $\{U_{ij}\}$ for $1 \le i \le t$ and $1 \le j \le s$ into a partition of $V$ with $st$ parts which adds at most $\frac{1}{st} ( e(G) - e(V_1 \cup \cdots \cup V_t))$ edges that are internal to the parts. Deleting these additional edges, we obtain a vertex partition of $G$ into $st$ parts from which we deleted at most $\frac{1}{st}\left(e(G)-e(V_1 \cup \cdots \cup V_t)\right)+\sum_{i=1}^t h(G[V_i],s)$ edges in order to make it $st$-partite. 
\end{proof}

We primarily focus on the case where we take disjoint vertex subsets $V_1, \ldots, V_k$ such that each $V_i$ is contained in the neighborhood of some vertex $v_i$. This yields a bound on $e(G) - e(V_1 \cup \cdots V_k)$ and an associated bound on $h(G, k)$.

\begin{cor}\label{c:neighbork}
If $G=(V,E)$ is a graph on $n$ vertices and $k$ is a positive integer, there are disjoint vertex subsets $V_1,\ldots,V_k$ such that each $V_i \subset N(v_i)$ for some vertex $v_i \in V$ and the number of edges of $G$ not in $G[V_1 \cup \cdots \cup V_k]$ is at most $\frac{n^2}{ek}$. 
\end{cor}
\begin{proof}
Apply Lemma \ref{cornewrandom} to obtain vertices $v_1,\ldots,v_k \in V$ such that for $U =  \bigcup_{i=1}^k N(v_i)$, $e(G) - e(U) \leq n^2/(ek)$. 
Define $V_1=N(v_1)$ and, for $i \geq 2$, define $V_i=N(v_i) \setminus \bigcup_{j<i}N(v_j)$, so $V_i \subset N(v_i)$. The sets $V_1,\ldots,V_k$ are disjoint and satisfy $V_1 \cup \cdots \cup V_k = U$, so the corollary clearly follows. 
\end{proof}

Via Corollary~\ref{c:neighbork} and Lemma~\ref{l:hpart}, we have the following immediate corollary.

\begin{cor}\label{c:partub}
If $G=(V,E)$ is a graph on $n$ vertices and $s$ and $t$ are positive integers, then there are disjoint vertex subsets $V_1,\ldots,V_t$ such that each $V_i$ is contained in the neighborhood of some vertex $v_i$ and 
$$h(G, st) \le \frac{n^2}{es^2t^2} + \sum_{i = 1}^t h(G[V_i], s).$$
\end{cor}

\subsection{$K_r$-Free Graphs}\label{s:clique}

The above results are helpful tools to bound $h(n, k, H)$ for a variety of families of forbidden subgraphs $H$. Here, we consider the case $H = K_r$ with $r \ge 3$. As a first application, Corollary~\ref{c:partub} immediately enables us to give an upper bound on $h(n, k, K_3)$.

\begin{prop}\label{p:k3} 
Any triangle-free graph on $n$ vertices can be made $k$-partite for $k \le n$  by deleting at most $n^2/ek^2$ edges, so $h(n, k, K_3) \le \frac{n^2}{ek^2}$. 
\end{prop}
\begin{proof}
Let $G$ be a triangle-free graph on $n$ vertices. Applying Corollary~\ref{c:partub} with $t = k$ and $s = 1$ implies that we can find vertices $v_1, \ldots v_k$ and disjoint vertex subsets $V_1, \ldots, V_k$ with $V_i \subset N(v_i)$ such that
$$h(G, k) \le \frac{n^2}{e k^2} + \sum_{i = 1}^k h(G[V_i], 1).$$
Since $G$ is triangle-free, $N(v_i)$ is an independent set and thus so is $V_i$, which implies that $h(G[V_i], 1) = 0$. Therefore, 
$h(G, k) \le \frac{n^2}{ek^2}$. Since this holds for all triangle-free graphs $G$ on $n$ vertices, we obtain the desired bound on $h(n,k,K_3)$. 
\end{proof}

It is helpful in further applications if the $V_i$ are not much larger than their average size. This can be obtained by furthering partitioning the large sets $V_i$ obtained in Corollary~\ref{c:neighbork}. 

\begin{cor}\label{c:evenparts}
If $G=(V,E)$ is a graph on $n$ vertices and $t \leq n$ is a positive integer, then there are disjoint vertex subsets $U_1,\ldots,U_{2t}$ such that each $U_j$ satisfies $|U_j| \leq \frac{n}{t}$ and is contained in the neighborhood of some vertex $u_j$, and $e(G)-e(U_1 \cup \cdots \cup U_{2t}) \leq \frac{n^2}{et}$. 
\end{cor}
\begin{proof}
By Corollary~\ref{c:neighbork}, there are vertex subsets $V_{1}, \ldots , V_{t}$, each a subset of a vertex neighborhood, where $V_1 \cup \cdots \cup V_t$ contains all but at most $n^2/et$ edges of $G$. Arbitrarily partition each $V_i$ into sets $U_j$ of size $\lfloor n/t \rfloor$, including if needed one set of size less than $\lfloor n/t \rfloor$. Thus, we obtain $a$ sets $U_j$ of size $\lfloor n/t \rfloor$ and $b$ sets $U_j$ of size strictly smaller than $\lfloor n/t \rfloor$, where $b \le t$. If $a<t$, then $a+b<2t$. Otherwise, $a \geq t$, and $t$ of these sets of size $\lfloor n/t \rfloor$ together have $t\lfloor n/t \rfloor>n-t$ elements, so less than $t$ elements are not in these $t$ sets. The remaining $a+b-t$ sets each have at least one element, so $a+b-t < t$ or equivalently $a+b < 2t$. We can add additional empty sets to make $2t$ total sets $U_j$, each of size at most $n/t$, and with $U_j \subset V_i \subset N(v_i)$ for some $i$.

\end{proof}

For a graph $H$ and vertex $v$, let $H_v$ denote the induced subgraph of $H$ formed by deleting $v$. We prove the following recursive upper bound on $h(n,k,H)$.

\begin{lemma}\label{l:divide}
If $s,t,n$ are positive integers, $H$ is a graph, and $v$ is a vertex of $H$ so that $H_v$ has no isolated vertices, then 
$$h(n,2st,H) \leq 2t \cdot h(n/t,s,H_v) + \frac{n^2}{2est^2}.$$
\end{lemma}
\begin{proof} 
Let $G$ be an $H$-free graph on $n$ vertices. By Corollary~\ref{c:evenparts}, there are $2t$ disjoint vertex subsets $U_1, \ldots U_{2t}$ such that each $U_i$ satisfies $|U_i| \leq \frac{n}{t}$ and is contained in the neighborhood of some vertex $u_i$. Further, we can pick the $U_i$ so that the number of edges of $G$ not in $G[U_1 \cup \cdots \cup U_{2t}]$ is at most $\frac{n^2}{et}$. As $G$ is $H$-free, then for all $u_i \in V(G)$, the induced subgraph $G[N(u_i)]$ is $H_v$-free for any vertex $v \in H$. Thus, for each $i$, $G[U_i]$ can be made $s$-partite by removing at most $h(|U_i|, s, H_{v}) \le h(n/t, s, H_v)$ edges in $G[U_i]$. 

For $1 \leq i \leq 2t$, we label the $s$ independent sets (after removing edges as above) which partition $U_i$  as $W_{i1}, \ldots W_{is}$. By Lemma~\ref{l:partpart}, 
we can then grow $\{W_{ij}\}_{1 \le i \le 2t, 1 \le j \le s}$ to a partition of $V(G)$ by adding at most $\frac{1}{2st} \left(e(G) - e(U_1 \cup \cdots \cup U_{2t})\right) \leq \frac{1}{2st}\cdot \frac{n^2}{et}=\frac{n^2}{2est^2}$ edges internal to the parts. Deleting these edges, we obtain the upper bound
$$h(n, 2st, H)\leq 2t \cdot h(n/t,s,H_v) + \frac{n^2}{2est^2}.$$
\end{proof}

Applying Lemma~\ref{l:divide} and induction on $r$, we can bound the number of edges to remove from a $K_r$-free graph $G$ on $n$ vertices so that the resulting subgraph is $k$-colorable. We first establish a bound on $h(n, k, K_r)$ when $k$ is a perfect $(r-2)^{\text{nd}}$ power of an even integer.

\begin{lemma} \label{l:cliquet}
For positive integers $n, k,$ and $r \ge 3$ so that $k = t^{r-2}$ for $t$ even, we have $$h(n, k, K_r) \le \alpha_r \cdot \frac{n^2}{k^{(r-1)/(r-2)}},$$
where $\alpha_r = \frac{5 \cdot 4^{r-3} - 2}{3e}$.
\end{lemma}
\begin{proof}
The proof is by induction on $r$. We proved the base case $r =3$ in Proposition~\ref{p:k3}. Let $s = k/t = t^{r-3}$. By the inductive hypothesis, we know that for all positive integers $n_0$,
$$h(n_0, s, K_{r-1}) \le \alpha_{r-1} \cdot \frac{n_0^2}{s^{(r-2)/(r-3)}}.$$
Note that the induced subgraph formed by deleting a vertex of $K_r$ is $K_{r-1}$. By Lemma~\ref{l:divide} (with parameter $t/2$ instead of $t$) and the above inequality (with $n_0 = 2n/t$), we have 
$$ h(n, k, K_r) \le \frac{2n^2}{est^2} +  t \cdot h \left( \frac{2n}{t}, s, K_{r-1} \right) \le \frac{2n^2}{est^2} +  t \alpha_{r-1} \frac{(2n/t)^2}{s^{(r-2)/(r-3)}} = \alpha_r \cdot \frac{n^2}{k^{(r-1)/(r-2)}},$$
where the equality uses $\alpha_r = \frac{2}{e} + 4\alpha_{r-1}$. This completes the proof.
\end{proof}

Lemma \ref{l:cliquet} establishes Theorem~\ref{t:clique} when $k$ is a perfect $(r-2)^{\text{nd}}$ power of an even integer. The following result is Theorem~\ref{t:clique} 
with an explicit constant factor.

\begin{thm}
For positive integers $n, k,$ and $r \ge 3$, we have $$h(n, k, K_r) \le \frac{5}{3} \cdot 4^{r-3} \cdot  \frac{n^2}{k^{(r-1)/(r-2)}}.$$
\end{thm}
\begin{proof}
Proposition \ref{p:k3} handles the case $r=3$, so we may assume $r \geq 4$. Recall that any graph on $n$ vertices can be made $k$-partite by removing at most $n^2/(2k)$ edges. Thus, if $k \leq (2r)^{r-2}$, as $4^{r-3} \geq  r \geq \frac{1}{2}k^{1/(r-2)}$, we have the desired inequality. We therefore suppose that $k > (2r)^{r-2}$. 

Let $\ell$ be the largest even perfect $(r-2)$-power which is at most $k$, so $\ell=\left(2 \left\lfloor \frac{k^{1/(r-2)}}{2} \right\rfloor \right)^{r-2}$.
By monotonicity and $\ell \leq k$, we have $h(n, k, K_r) \le h \left(n,\ell, K_r \right)$.
Applying Lemma~\ref{l:cliquet},
$$h(n,\ell, K_r) \le \alpha_r \frac{n^2}{\ell^{(r-1)/(r-2)}}, \quad \textrm{where}~\alpha_r = \frac{5 \cdot 4^{r-3} - 2}{3e}.$$
Since $k > (2r)^{r-2}$, it follows that $k \leq \left(\frac{2(r+1)}{2r}\right)^{r-2} \ell =\left(1+\frac{1}{r}\right)^{r-2} \ell$. It follows that $$k^{(r-1)/(r-2)} \leq \left(1+\frac{1}{r}\right)^{r-1}\ell^{(r-1)/(r-2)} \leq e \ell^{(r-1)/(r-2)}.$$ 
Substituting, we obtain $$h(n,k,K_r) \leq \frac{5}{3} \cdot 4^{r-3} \cdot \frac{n^2}{k^{(r-1)/(r-2)}}.$$

\end{proof}

\section{Odd Cycle-Free Graphs}\label{s:oddcycle}
In this section we study how few edges we can remove from any graph on $n$ vertices with a fixed forbidden odd cycle to make it $k$-colorable. 

\subsection{Upper bounds for odd cycles}
We begin by tackling a simpler problem, bounding how far a graph of large odd girth is from being $k$-colorable. That is, we first consider $h(n, k, \mcH)$ where $\mcH$ is the family of odd cycles of length at most $2r+1$. We later show that this number is close to $h(n,k,C_{2r+1})$. 

We prove that graphs of large odd girth contain an independent set $B$ with relatively many edges incident to $B$. We repeatedly apply this to pull out $k$ disjoint independent sets $B_1, ..., B_k$ such that the remaining induced subgraph contains relatively few edges. By Lemma \ref{l:partpart}, we can grow these $k$ independent sets into a $k$-partition of the vertex set so that 
few edges are internal to the parts.

The following definitions will be helpful. 

\begin{defn}
Given a graph $G = (V, E)$ and a vertex $v \in V$, the \textit{$i$th neighborhood} of $v$, denoted by $N_i(v)$, is the set of vertices in $V$ of distance exactly $i$ from $v$.
\end{defn}

For example, $N_0(v) = \{v\}$, $N_1(v) = N(v)$, and $N_2(v)$ is the set of vertices in $V \backslash (\{v\} \cup N(v))$ that have a neighbor in $N(v)$. 
For a vertex subset $T$, let $N(T)$ denote the set of vertices in $V \setminus T$ adjacent to at least one vertex in $T$. For a graph $G$ and vertex subsets $S$ and $T$, let $e(S,T)$ be the number of pairs in $S \times T$ that are edges of $G$.  

\begin{defn}
For a graph $G=(V,E)$ and $S \subset V$, let $D(S) = e(S, V)$ be the sum of the degrees of vertices in $S$. 
\end{defn}

Note that $D(S)$ counts the edges contained in $S$ twice and the edges with exactly one endpoint in $S$ once. It is a useful measure of the number of edges that contain a vertex in $S$.

We first show that for any graph $G = (V,E)$ of large odd girth
 and any subset $S\subset V$, there is an independent set $B\subset S$
with poor edge expansion into $S$. Removing such $B$ and its neighborhood and iteratively applying the argument will give an independent set $A$ (the union of the $B$'s) with comparatively large $D(A)$.
In the following lemmas, we take a graph $G = (V, E)$ on $n$ vertices, $r$ a positive integer, and a fixed $S \subset V$. We let $$x:=\left(\frac{|S|n}{D(S)}\right)^{1/r}.$$  

\begin{lemma} \label{l:gB}
Let $G=(V,E)$ be a graph of odd girth larger than $2r + 1$ and $S \subset V$. There exists an independent set $B \subset S$ such that  $$D(B) \ge \frac{D(N(B) \cap S)}{x + 1}.$$ 
\end{lemma}
\begin{proof}
Pick $v \in V, u \in S$ uniformly at random, so $$\Prob(u \in N_1(v)) = \frac{\mathbb{E}_{u \in S}[d(u)]}{n}.$$ If $u \in N_1(v) \cap S,$ then it contributes $d(u)$ to $D(N_1(v) \cap S)$. Therefore,
$$\EE_v [D(N_1(v) \cap S)] = \EE_v\left[ \sum_{u \in S} d(u) \Prob(u \in N_1(v)) \right] = \sum_{u \in S} \EE_v \left[ d(u) \cdot \frac{d(u)}{n} \right] = \sum_{u \in S} \frac{d(u)^2}{n}.$$
Hence, by picking $v \in V$ such that $D(N_1(v) \cap S)$ is maximized, 
$$D(N_1(v) \cap S) \ge \sum_{u \in S} \frac{d(u)^2}{n} \overset{(*)}\ge \frac{\left( \sum_{u \in S} d(u) \right)^2}{n|S|} = \frac{D(S)^2}{n|S|},$$
where $(*)$ follows by the Cauchy-Schwarz inequality. Let $N_i = N_i(v) \cap S$. 
Note $N_1, ..., N_r$ are all independent sets. Indeed, if some $N_i$ for $1 \le i \le r$ contained an edge $e = (v_1, v_2)$, then $v \rightarrow \cdots \rightarrow v_1 \overset{e}\rightarrow v_2 \rightarrow \cdots \rightarrow v$ is an odd walk of length $2i + 1$ and thus contains an odd cycle of length at most $2r+1$. This contradicts $G$ having odd girth larger than $2r + 1$.

We next study the growth of $D(N_i)$. If $D(N_2) < x D(N_1)$,
then
$$D(N(N_1) \cap S) \le |N_1| + D(N_2) < |N_1| + x D(N_1) \le (x + 1) D(N_1),$$
which implies that
$$D(N_1) > \frac{D(N(N_1) \cap S)}{x + 1},$$
so we can take $B = N_1$ to satisfy the lemma. 
Else, for $i = 2, ..., r-1$, if $D(N_j) \ge x D(N_{j-1})$, for all $2 \leq j < i$, and $D(N_i) < x D(N_{i-1})$, we similarly find that
$$D(N(N_i) \cap S) \le D(N_{i-1}) + D(N_{i+1}) < \frac{D(N_i)}{x} + x D(N_i),$$
and hence 
$$D(N_{i}) > \frac{D(N(N_{i})\cap S)}{x+1/x} \ge \frac{D(N(N_{i})\cap S)}{x+1},$$
so we can take $B = N_{i}$ to satisfy the lemma. 
If none of $N_1, \ldots, N_{r-1}$ satisfy the conditions of the lemma statement as a subset $B \subset S$, then 
$$D(N_r) \ge x^{r-1} \frac{D(S)^2}{|S|n} = \frac{D(S)}{x}.$$
This implies that we can pick $B = N_r = N_r(v) \cap S$ to satisfy the lemma.
\end{proof}

We can use this lemma to establish a more helpful result in the same direction. 

\begin{lemma}\label{l:g2} 
Let $G = (V, E)$ be a graph of odd girth larger than $2r + 1$ and $S \subset V$. There exists an independent set $A \subset S$ such that $D(A) \ge D(S)/8x$. 
\end{lemma}
\begin{proof}
We use Lemma~\ref{l:gB} to pull out independent sets one at a time. By deleting their neighborhoods and repeating, we construct a large independent set $A$ which is the union of these independent sets and show that $A$ has the desired properties. 

Let $V_1 = V$ and  $S_1 = S$. We apply Lemma~\ref{l:gB} to obtain an independent set $B_1 \subset S_1$ such that $$D(B_1) \ge \frac{D(N(B_1) \cap S_1)}{x_1 + 1}, \quad x_1 = \left( \frac{|S_1|n}{D(S_1)} \right)^{1/r}.$$ We repeatedly apply Lemma~\ref{l:gB} to $G[V_i]$ and $S_i$, letting $$V_i = V_{i-1} \backslash (B_{i-1} \cup N(B_{i-1})),\quad S_i = S_{i-1} \cap V_i.$$ At each iteration, we obtain an independent set $B_i \subset S_i$ such that $$D(B_i) \ge \frac{D(N(B_i) \cap S_i)}{x_i + 1}, \quad x_i = \left( \frac{|S_i|n}{D(S_i)} \right)^{1/r}.$$

By construction, by step $i$ we have deleted $B_{i-1}$ and its neighbor set $N(B_{i-1})$, so $\bigcup_{j < i} B_j$ is an independent set. We continue the construction described above as long as $D(S_i) \ge D(S)/2$. Suppose we construct $s$ independent sets in total through this process. Then $D(S_s) < D(S)/2$, but $D(S_i) \ge D(S)/2$ for $i < s$. 
Let $A = \bigcup_{i = 1}^s B_i$ be the resulting large independent set.
We can bound
$$x_i = \left( \frac{|S_i|n}{D(S_i)}\right)^{1/r} \le \left( \frac{|S|n}{D(S_i)}\right)^{1/r} \le \left( \frac{|S|n}{D(S)/2}\right)^{1/r} \le 2^{1/r}x.$$
From above, we have 
$$D(B_i) \ge \frac{D(N(B_i) \cap S_i)}{x_i + 1} \overset{(*)}\ge  \frac{D(N(B_i) \cap S_i) + D(B_i)}{x_i + 2} \ge \frac{D(N(B_i) \cap S_i) + D(B_i)}{2x + 2}. $$
where $(*)$ follows since $D(B_i) \ge D(B_i)/1$ and if $x \ge a/b, c/d$ then $x \ge (a+c)/(b+d)$. 
This allows us to bound $D(A)$ as
\begin{equation}\label{e:da1}
D(A) = \sum_{i = 1}^s D(B_i) \ge  \sum_{i = 1}^s \frac{D(N(B_i) \cap S_i) + D(B_i)}{2x + 2} = \frac{D(A) + D(N(A) \cap S)}{2x + 2},
\end{equation}
where the last equality follows since $A$ is the union of the disjoint sets $B_1,\ldots,B_s$, $N(A) \cap S$ is the union of the disjoint sets $N(B_1) \cap S_1, \ldots, N(B_s) \cap S_s$, and 
$D$ is additive on the union of disjoint sets. 
Further,
\begin{align*}
D(S) &= D(A) + D(N(A) \cap S) + D(S \backslash (A \cup N(A))) \\
&= D(A) + D(N(A) \cap S) + D(S_s),
\end{align*}
and, as $D(S_s) < D(S)/2$, 
\begin{equation}\label{e:da2}
D(A) + D(N(A) \cap S) \ge \frac{D(S)}{2}.
\end{equation}
Combining~\eqref{e:da1} and~\eqref{e:da2} gives the desired bound:
$D(A) \ge D(S)/(4x+4) \ge D(S)/8x $.
\end{proof}

We next use this lemma to obtain an upper bound on $h(G,k)$ for graphs $G$ with no short odd cycles.

\begin{proof}[Proof of Theorem~\ref{t:cyclegirth}]
We iteratively apply Lemma~\ref{l:g2} to obtain $k$ disjoint independent sets which are each incident to many edges. Let $S_1 = V$, and let $A_1$ be a subset of $S_1$ with the properties guaranteed by Lemma~\ref{l:g2}. Proceed for $k$ iterations, letting $S_i = V \backslash \bigcup_{j = 1}^{i-1} A_j$, to obtain $A_i \subset S_i$ per Lemma~\ref{l:g2} with large $D(A_i)$. By construction, $S_{k+1} = V \backslash \bigcup_{i = 1}^{k} A_i$, the sets $A_1, ..., A_k$ are independent, and for each $A_i$, 
 $$D(A_{i}) \geq \frac{D(S_i)}{8x_i}, \quad x_i=\left(\frac{n |S_i|}{D(S_i)}\right)^{\frac{1}{r}}.$$

By Lemma \ref{l:partpart}, we can assign each $v \in S_{k+1}$ to one of the $A_i$ and can make the graph $k$-partite by deleting the at most $D(S_{k+1})/k$ edges in parts. 

By construction, we have the recursive upper bound 
$$D(S_{i+1}) = D(S_i) - D(A_i) \le \left(1 - \frac{1}{8x_i} \right) D(S_i).$$
Let 
$\delta_i := \frac{D(S_i)}{n^2}.$
Since $|S_i| \le n$, the above relation yields the recursive inequality
\begin{equation}\label{e:recur}
\delta_{i+1} \le \delta_i \left( 1 - \frac18 \delta_i^{1/r} \right).
\end{equation}
Inequality~\eqref{e:recur} implies that if $\delta_l > \eps/2$ for some $\eps > 0$, then $$\delta_{l+1} - \delta_l < \frac{-\delta_l}{8} \left(\frac{\eps}{2} \right)^{1/r}.$$ 
If $\delta_i \le \eps \le 1$ and $\delta_j > \eps/2$ for $j > i$, then
$$\delta_j - \delta_i = \sum_{l = i}^j (\delta_{l + 1} - \delta_{l} ) < \sum_{l = i}^j \frac{- \delta_l}{8} \left( \frac{\eps}{2} \right)^{1/r} < \frac{(i-j)\delta_{i}}{16} \left( \frac{\eps}{2} \right)^{1/r}.$$
Thus, for $j - i \ge  8 \left( 2/\eps \right)^{1/r},$ 
$$\delta_j - \delta_i \le \frac{-8 \left(2/\eps\right)^{1/r}\delta_{i}}{16} \left( \frac{\eps}{2} \right)^{1/r} = \frac{-\delta_i}{2},$$
which yields
$\delta_j \le \delta_i/2 \le \eps/2.$
\newline \newline 
Note that $\delta_1 \le 1$. Let
$$u = \left\lfloor r \log_2 \left( \frac{k \ln 2}{32r} \right) \right \rfloor - 1.$$
We show that $\delta_{k+1} \le 1/2^u$. If $u < 0$, then we have $\delta_{k+1} \le \delta_1 \le 1 \le 1/2^u$. So we can suppose that $u \ge 0$. 
Using the above bound on the decay of $\delta_i$, letting $\eps = 2^{-i}$ for $i = 0, 1, ..., u-1$, we note that $\delta_{j} \le 1/2^u$ for $j := 8 \sum_{i = 0}^{u - 1} \left\lceil \left( \frac{2}{2^{-i}} \right)^{1/r} \right\rceil $, since
\footnotesize
$$j = 8 \sum_{i = 0}^{u - 1} \left\lceil \left( \frac{2}{2^{-i}} \right)^{1/r} \right\rceil = 16 \sum_{i = 0}^{u - 1} \left(2^{1/r}\right)^{i+1} = 16 \cdot 2^{1/r} \cdot \left(\frac{2^{(u+1)/r}-1}{2^{1/r}-1}\right) \overset{(*)}<  \frac{32(2^{(u+1)/r})}{(\ln 2)/r} \le  \frac{32r \left( \frac{\ln 2}{32r} \cdot k \right)}{\ln 2} < k + 1,$$
\normalsize
where $(*)$ follows since $e^x \ge 1 + x$ for all $x$ (so $2^{1/r} - 1 = e^{\ln 2/r} - 1 > \ln 2/r$). 
For this choice of $u$ we have that
$$\delta_{k+1} \le \frac{1}{2^u} \le 2^{1 - \lfloor r \log_2 \left( \frac{\ln 2}{8r} k \right) \rfloor} \le 4 \left( \frac{k \ln 2}{8r} \right)^{-r} = \frac{ 4\left( \frac{8r}{\ln 2} \right)^r }{k^r}.$$
This gives the desired bound on $h(G, k)$ with $c_r= 4 \left( 8r/\ln 2 \right)^r$:
$$h(G, k) \le \frac{D(S_{k+1})}{k} \le \frac{\delta_{k+1}n^2}{k} \le \frac{c_rn^2}{k^{r+1}} < \frac{4\left(12r\right)^{r}n^2}{k^{r+1}}.$$
\end{proof}
Theorem~\ref{t:cyclegirth} gives a bound on $h(G, k)$ when $G$ has odd girth larger than $2r+1$, which, as we show in the next subsection, is tight up to a factor depending only on $r$.
Our goal is to understand a less constrained family of graphs, those with a single fixed forbidden odd cycle. To do this, the following lemma shows that if a graph has a forbidden 
odd cycle $C_{2r+1}$, then we can delete a small number of edges to get rid of the next shorter odd cycles.  

\begin{lemma}\label{l:removecycle}
If a graph $G = (V, E)$ on $n$ vertices is $C_{2r+1}$-free, then $G$ can be made to have odd girth larger than $2r+1$ by removing $O_r(n^{3/2})$ edges.
\end{lemma}
\begin{proof}
For each odd integer $1<\ell<2r+1$, fix a maximal collection $\mathcal{C}_{\ell}$ of edge-disjoint copies of $C_{\ell}$ in $G$. Suppose we have $t_{\ell}$ such copies. To remove all 
copies of $C_{\ell}$ from $G$, we must delete at least $t_{\ell}$ edges (at least one edge in each edge-disjoint $C_{\ell}$). If we delete all $\ell t_{\ell}$ edges in these cycles, the resulting graph is $C_{\ell}$-free. Thus, the minimum number of edges to delete to make the graph $C_{\ell}$-free is within a factor $\ell$ of the size of any maximal collection of edge-disjoint copies of $C_{\ell}$. 

Consider a random subset $U \subset V$ formed by including each element with probability $p=1/\ell$ independently of the other vertices. Call an edge of a cycle in $\mathcal{C}_{\ell}$ {\it special} if both of its vertices are in $U$ and no other vertex of the cycle is in $U$. The probability that a given edge of a cycle in $\mathcal{C}_{\ell}$ is special is $p^2 (1-p)^{\ell-2}>1/(e\ell^2)$. Hence, by linearity of expectation, the expected number of special edges is at least $\ell t_{\ell}/(e\ell^2)=t_{\ell}/(e\ell)$. So we can fix a subset $U$ with at least $t_{\ell}/(e\ell)$ special edges. 

Let $2d:=2r+3-\ell$, so $d \in [2,r]$ is an integer as $\ell$ is odd. As the graph $G$ is $C_{2r+1}$-free, there is no cycle of length $2d$ of special edges. Indeed, suppose there is a cycle $C'$ of length $2d$ of special edges, and let $e$ be an edge of this cycle. Edge $e$ is by definition in a cycle $C''$ of length $\ell$ with none of its other vertices in $U$. So gluing together $C'$ and $C''$ and deleting the common edge $e$, we obtain a cycle of length $(2r+3-\ell)+\ell-2=2r+1$, contradicting the assumption that $G$ is $C_{2r+1}$-free. 

Recall that the \textit{extremal number} $\textrm{ex}(n,H)$ is the maximum number of edges an $H$-free graph on $n$ vertices can have. So the number of special edges is at most $$ \textrm{ex}(|U|,C_{2d}) \leq \textrm{ex}(n,C_{2d}) \leq 8(d-1)n^{1+1/d},$$
where the last bound is due to Verstra\"ete \cite{VE00}\footnote{There is a long history of bounding the extremal number of even cycles, including by Erd\H{o}s \cite{ER63}, Bondy and Simonovits \cite{BO74}, and most recently improvements by Pikhurkho \cite{PI12} and further by Bukh and Jiang \cite{BU17}.}. Hence, $t_{\ell}/(e\ell) \leq 8(d-1)n^{1+1/d}$, or equivalently,  
$\ell t_{\ell} \leq \ell^2 e 8(d-1)n^{1+1/d}$. So the number of edges we can delete 
from $G$ to make the resulting subgraph have odd girth larger than $2r+1$ is at most 
$$\sum_{3 \leq \ell \leq 2r-1,~\ell~\textrm{odd}} \ell t_{\ell} \leq \sum_{d=2}^{r} \ell^2 e 8(d-1)n^{1+1/d} \leq \sum_{d=2}^{r} 32e r^2d n^{1+1/d} \leq 100r^4n^{3/2}.$$ 

\end{proof}

From Lemma~\ref{l:removecycle} and Theorem~\ref{t:cyclegirth}, we immediately obtain Theorem~\ref{t:oddcycle}. 

We remark that Conlon, Fox, Sudakov, and Zhao \cite{CFSZ} recently improved the bound in Lemma \ref{l:removecycle} for $r=2$ to $o(n^{3/2})$.

\subsection{Lower bounds for odd cycles}  
In this subsection, we give a construction which shows that Theorem~\ref{t:oddcycle} is tight for sufficiently large $n$ to within a factor only depending on the length of the forbidden odd cycle. The construction is based on a construction of Alon~\cite{AL94} (see the discussions in \cite{AL98} and \cite{KRI06}) of a rather dense pseudorandom graph of large odd girth. 

\begin{defn}
An $(n, d, \lambda)$-\textit{graph} is a $d$-regular graph $G$ on $n$ vertices such that the second largest in absolute value eigenvalue has magnitude at most $\lambda$.
\end{defn}

To discuss the properties of this construction, we first recall the expander mixing lemma, a classical result in spectral graph theory. An early version is due to Alon and Chung \cite{AL88}. 

\begin{lemma}[Expander Mixing Lemma] \label{t:expander}
If $G=(V,E)$ is an $(n,d,\lambda)$-graph and $A,B \subset V$, then 
$$ \left|e(A, B) - \frac{d}{n}|A||B| \right| \le \lambda \sqrt{|A||B|}.$$
\end{lemma}

Note that if a graph $G$ on $n \le n'$ vertices is $\mcH$-free and no graph in $\mcH$ has isolated vertices, by adding $n' - n$ dummy vertices we can obtain a graph $G'$ on $n'$ vertices, where $h(n', k, \mcH)  \ge h(G, k)$. Since this holds for all such $G$ on $n$ vertices, we have the following observation, which will be relevant to our future discussion of blow-ups of a fixed graph.
\begin{prop} \label{p:monotone} 
Given a family of graphs $\mcH$, if no graph in $\mcH$ has isolated vertices and $n' \ge n$, then
$$h(n', k, \mcH) \ge h(n, k, \mcH).$$
\end{prop}

The results which follow proceed towards the goal of showing that $h(n,k,C_{2r+1})=\Omega_r(n^{2}/k^{r+1})$.

We first note that for arbitrary $G = (V, E)$, we can compute $h(G[t], k)$ in terms of $h(G, k)$, where $G[t]$ is the \textit{$t$-blow-up of $G$}, the graph on $t|V|$ vertices given by the lexicographic product of $G$ with an empty graph on $t$ vertices. 

\begin{lemma}\label{l:blowup}
Let $G[t]$ be the $t$-blow-up of $G = (V, E)$. Then,
$$h(G[t], k) = t^2 h(G, k).$$
\end{lemma}
\begin{proof}
By taking the blow-up of any vertex partition of $G$ into $k$ parts, we see that $$h(G[t], k) \le t^2 h(G, k).$$
To complete the proof, we next show the reverse inequality. Consider a vertex partition $P$ of $G[t]$ into $k$ parts. 
Consider a copy of $G$ in $G[t]$ with exactly one vertex in each of the $|V|$ parts of order $t$. Each such copy has at least $h(G, k)$ of its edges inside parts of $P$. The number of such copies of $G$ is $t^{|V|}$ and each edge of $G[t]$ is in exactly $t^{|V|-2}$ such copies of $G$. Thus, at least $h(G, k)t^{|V|}/t^{|V|-2}=t^2 h(G, k)$ edges of $G$ must be inside parts of $P$. Hence, $h(G[t], k) \geq t^2 h(G, k)$.
\end{proof}

It will be helpful to introduce the following definition. 
\begin{defn}
A family $\mcH$ of graphs is \textit{closed under homomorphism} if for any $H \in \mcH$ and graph homomorphism $\phi: H \rightarrow H'$, graph $H'$ is in $\mcH$.
\end{defn}

We want to establish a result similar to Lemma~\ref{p:monotone} for $t$-blow-ups.

\begin{lemma}
If $\mcH$ is closed under homomorphism and $t$ is a positive integer, then
$$h(tn, k, \mcH) \ge t^2 h(n, k, \mcH).$$
\end{lemma}
\begin{proof}
Consider a graph $G$ on $n$ vertices which is $\mathcal{H}$-free. Its $t$-blow-up $G[t]$ is also $\mathcal{H}$-free. Hence, $h(G[t],k)\leq h(tn,k,\mathcal{H})$. Furthermore, by Lemma~\ref{l:blowup}, we see that $h(G[t],k)=t^{2}h(G,k)$ and therefore $t^{2}h(G,k)\leq h(tn,k,\mathcal{H})$ holds for any $\mathcal{H}$-free graph $G$. 
By taking the maximum over the left hand side of this inequality, we obtain the desired inequality.
\end{proof}

We use the following result, extending Alon's construction of a pseudorandom triangle-free graph which is as dense as possible to one of large odd girth~\cite{KRI06, AL98}. 

\begin{lemma}[\S 3,~\cite{AL98}] \label{t:girth} 
For each positive integer $r$ there is $c_r$ such that the following holds. For every integer $a \geq 2$ and $N=2^{(2r+1)a}$, there is an $(N, d, \lambda)$-graph $G$ of odd girth larger than $2r+1$ such that $d \geq \frac{1}{8}N^{2/(2r+1)}$ and $\lambda \le c_r\sqrt{d}$.
\end{lemma}

We use this lemma to get a lower bound on $h(n, k, \mcH)$ with $\mcH = \{C_3, C_5, \ldots, C_{2r+1}\}$.

\begin{proof}[Proof of Theorem~\ref{c:lb}]

Let $a=\lceil \log_2(2 c_r k) \rceil$, where $c_r$ is chosen as in Lemma~\ref{t:girth}, and $N=2^{(2r+1)a}$. By Lemma~\ref{t:girth}, there is a $(N, d, \lambda)$-graph $H$ with $d \ge \frac18 N^{2/(2r+1)}$ and $\lambda \leq c_r\sqrt{d}$. 
For $n \geq N$, let $t = \lfloor n/N \rfloor$ and let $n_0 = N t$, so $n_0 \leq n$. Let $G = H[t]$ be the balanced $t$-blow up of $H$, so $|V(G)| = n_0$. 
Since $H$ has odd girth larger than $2r+1$, $G$ also has odd girth larger than $2r+1$. Let 
$\mcH = \{C_3, C_5, \ldots, C_{2r+1}\}$. By Proposition~\ref{p:monotone} and Lemma~\ref{l:blowup}, we have 
$$h(n, k, \mcH) \ge h(n_0, k, \mcH) \ge h(G, k) = h(H[t], k) = t^2 h(H, k).$$
We give a lower bound on $h(H, k)$ which implies the desired lower bound on $h(n, k, \mcH)$.
Consider a $k$-partition $V(H) = V_1 \sqcup \cdots \sqcup V_k$ that minimizes $\sum_{i = 1}^k e(V_i)$, which is the minimum number of edges to delete 
from $H$ to obtain a $k$-colorable subgraph. By Lemma~\ref{t:expander}, we have
\begin{eqnarray*}2\sum_{i = 1}^k e(V_i) & = & \sum_{i = 1}^k e(V_i, V_i) \ge \sum_{i = 1}^k \left( \frac{d}{N}|V_i |^2 - \lambda|V_i| \right) = \frac{d}{N} \sum_{i = 1}^k |V_i|^2 - \lambda N \overset{(*)}\ge \frac{dN}{k} - \lambda N \overset{(**)}\ge \frac{dN}{2k} \\ & \ge &  \Omega_r \left( N^2/k^{r+1} \right).\end{eqnarray*}
Here $(*)$ follows by convexity of $f(x) = x^2$ and $(**)$ follows from $\lambda \le c_r\sqrt{d} \leq \frac{d}{2k},$ 
which in turn follows from $N \geq (2c_rk)^{2r+1}$ and $d \ge \frac18 N^{2/(2r+1)}$. Hence, 
$h(n, k, \mcH) \ge t^2 h(H, k) = \Omega_r \left( n^2/k^{r+1} \right)$, which completes the proof. 
\end{proof}

We can strengthen Theorem~\ref{c:lb} when $H$ is a triangle. Guo and Warnke~\cite{GUO20} show the existence of triangle-free graphs with discrepancy like random graphs by using a semi-random variant of the triangle-free process (studied in~\cite{BOH99}).

\begin{thm}[Theorem 4,~\cite{GUO20}]\label{t:trianglefree}
There exist $\beta_0, D_0 > 0$ such that for all $\gamma, \delta \in (0, 1], \beta \in (0, \beta_0)$ and $C \ge D_0/(\delta^2\sqrt{B}\gamma)$, the following holds for all $n \ge n_0(\gamma, \delta, \beta, C)$ with $\rho := \sqrt{\beta\log n/n}$: for any $n$-vertex graph $G'$, there exists a triangle-free subgraph $G \subset G'$ on the same vertex set such that 
$$e_{G}(A, B) = (1 \pm \delta)\rho e_{G'}(A, B)$$
for all vertex-sets $A, B \subset V(G')$ with $|A| = |B| = \lceil C \sqrt{n\log n}\rceil$ and $e_G'(A, B) \ge \gamma|A||B|$.
\end{thm}

This yields an improved lower bound on $h(n, k, C_3)$.

\begin{prop}\label{p:cyclefree}
There exists an absolute constant $c$ such that for all $n \ge ck^2 \log k$, there exists a triangle-free graph $\Gamma$ on $n$ vertices with $h(\Gamma, k) = \Omega(n^2/k^2)$.
\end{prop}
\begin{proof}
We apply Theorem~\ref{t:trianglefree} with $G' = K_N$, $\gamma = 1$, and $\delta = 1/2$. This yields a triangle-free graph $G \subset K_N$ such that for $A \subset V$ with $|A| = \lceil C \sqrt{N\log N}\rceil$, we have that $$e_{G}(A) \ge \frac12 \sqrt{\beta \log n / N} \cdot |A|^2,$$
i.e. $G$ resembles a random graph with edge probability $\Theta(\sqrt{\log N/ N})$ and $\alpha(G) \le 2 \lceil C \sqrt{N\log N}\rceil$.
We choose $N = c k^2 \log k$ for sufficiently large absolute constant $c$. Any $k$-partition of $G$ as above has relatively dense parts. Precisely, we have at least $\frac14 \sqrt{\beta \log N / N} \cdot \frac{N^2}{k}$ edges internal to any $k$-partition. Then, taking a balanced blowup of $G$ (in which each part has either $\lfloor n/N \rfloor$  or $\lceil n/N \rceil$ vertices) gives by Lemma~\ref{l:blowup} a triangle-free graph $\Gamma$ on $n$ vertices with $h(\Gamma, k) = \Omega(n^2/k^2)$.
\end{proof}

Note that Proposition~\ref{p:cyclefree} is essentially best possible. This follows from the construction of Kim~\cite{KIM95}, refined by Fiz Pontiveros, Griffiths and Morris~\cite{FIZP} of triangle-free graphs on at most $(4 + o(1))k^2 \log k$ vertices with chromatic number at most $k$. 

\section{Applications to Other Forbidden Subgraphs}\label{s:graphhom}

So far, we have only proved bounds on $h(n, k, H)$ when $H$ is an odd cycle or clique. In this section, we obtain bounds for a broader class of graphs $H$.  In  particular, we prove that if $H'$ is a subgraph of a fixed graph $H$, and $H$ has a homomorphism to $H'$, then $h(n,k,H)$ is within $o(n^2)$ of $h(n,k,H')$. 

\subsection{Graph Homomorphisms}\label{s:homgen}

To obtain these results, we use the graph removal lemma, which first appeared in 
\cite{ADLRY,Fur94}. It extends the triangle removal lemma of Ruzsa and Szemer\'edi (see the survey \cite{CON13} for details). 

\begin{thm}[Graph removal lemma]
\label{t:grl}
For any graph $H$ on $h$ vertices and any $\varepsilon > 0$, there exists $\delta > 0$ such that any graph on $n$ vertices that contains at most $\delta n^{h}$ copies of $H$ can be made $H$-free by removing at most $\varepsilon n^2$ edges.
\end{thm}

Recall that a homomorphism from a graph $H$ to a graph $H'$ is a (not necessarily injective) map $\rho: V(H) \rightarrow V(H')$ that maps edges of $H$ to edges of $H'$. 
We also use the following lemma of Erd\H{o}s. 

\begin{lemma}[\cite{Er64}]\label{Erdoshyper}
For $\delta > 0$, $r \ge 2$, $t \ge 1$, and sufficiently large $n$, every $r$-uniform hypergraph $\Gamma$ on $n$ vertices with at least $\delta n^r$ edges contains a complete $r$-partite, $r$-uniform subhypergraph with parts of order $t$.
\end{lemma}

\begin{thm}\label{t:hom}
Suppose $\mathcal{H}$ and $\mathcal{H}'$ are fixed finite families of graphs such that for each $H' \in \mathcal{H}'$, there is some $H \in \mathcal{H}$ such that $H$ has a homomorphism to $H'$. If $k$ is a fixed positive integer, then 
$$h(n, k, \mathcal{H}) \le h(n, k, \mathcal{H}')+o(n^2).$$
\end{thm} 
\begin{proof}
Let $H'$ be a graph in $\mathcal{H}'$ and $H$ be some graph in $\mathcal{H}$ for which $H'$ has a homomorphism to $H$. Let $r$ denote the number of vertices of $H'$, and $t$ denote the number of vertices of $H$. Label the vertices of $H'$ as $\{1,\ldots,r\}$. Fix any $\varepsilon>0$ and let $\delta>0$ be as in the graph removal lemma for $H'$. Let $G$ be an $\mathcal{H}$-free graph on $n$ vertices.  Consider the $r$-uniform $r$-partite hypergraph $X$ with parts $V_1,\ldots,V_r$ with each $V_i$ a copy of $V(G)$, and $(v_1,v_2,\ldots,v_r) \in V_1 \times V_2 \times \cdots \times V_r$ is an edge of $X$ if there is a copy of $H'$ with $v_i$ a copy of $i$ for $i \in \{1,\ldots,r\}$. 

If there are at least $\delta n^r$ copies of $H'$ in $G$, then $X$ contains at least $\delta n^r=\delta t^{-r}|V(X)|^r$ edges. As we may assume $n$ is sufficiently large, Lemma \ref{Erdoshyper} implies that $X$ contains a copy of the complete $r$-partite $r$-uniform hypergraph with parts of order $t$. As $H$ has a homomorphism to $H'$, we can then find a copy of $H$ with vertices among the vertices of the copy of the complete $r$-partite $r$-uniform hypergraph with parts of order $t$, contradicting that $G$ is $\mcH$-free. 

So we may suppose $G$ has less than $\delta n^r$ copies of $H'$. By the graph removal lemma applied to $H'$, we can remove $\varepsilon n^2$ edges from $G$ to make it $H'$-free. We can do this edge removal for each $H' \in \mathcal{H}'$ to make the graph $\mathcal{H}'$-free. We can then remove an additional $h(n,k,\mcH')$ edges to make it $k$-partite. We have thus obtained the desired upper bound on $h(n,k,\mcH)$. 
\end{proof}

We have the following immediate corollary by taking $\mathcal{H}$ and $\mathcal{H}'$ to each consist of a single graph. 

\begin{cor}\label{cor:homsingle}
If $H$ and $H'$ are fixed graphs for which $H$ has a homomorphism to $H'$ and $k$ is a fixed positive integer, then
$h(n, k, H) \le h(n, k, H')+o(n^2)$.
\end{cor}

In the special case that $H$ is not bipartite, $H$ and $k$ are fixed, and $H'$ is a subgraph of $H$, we get that these parameters are asymptotically equal, as in Theorem~\ref{c:homomorphism}.

\begin{proof}[Proof of Theorem~\ref{c:homomorphism}]
The upper bound follows from Corollary~\ref{cor:homsingle}. As $H'$ is a subgraph of $H$, then any $H'$-free graph is also $H$-free, and hence $h(n, k, H) \le h(n, k, H')$.
\end{proof} 

\begin{rem}
Note that Theorem~\ref{c:homomorphism} immediately implies a weaker version of Theorem~\ref{t:oddcycle}, namely that for $\mcH' = \{C_3, C_5, \ldots C_{2r-1}, C_{2r+1}\}$,
$$h(n, k, C_{2r+1}) = h(n, k, \mcH') + o_r(n^2).$$
\end{rem}

As another example, let $K_{2,2,2}$ be the complete tripartite graph on six vertices with two vertices in each part. From Theorems~\ref{t:hom}~and~\ref{c:homomorphism}, we have $$h(n, k, K_3) \le h(n, k, K_{2,2,2}) \le h(n, k, K_3)+o(n^2).$$ Consequently, although we don't know for any fixed $k \geq 2$ the asymptotic value of $h(n,k,K_3)$ or $h(n,k,K_{2,2,2})$, we know that they are asymptotically the same.

\subsection{Forbidding a wheel}
Recall that the \textit{wheel} $W_{l}$ is the graph of $l + 1$ vertices consisting of an $l$-cycle and an additional vertex adjacent to all of the vertices of the $l$-cycle.
The above results on graph homomorphisms allow us to bound how many edges we need to delete to make a $W_{l}$-free graph on $n$ vertices $k$-colorable. We first get an asymptotic answer for even wheels $W_l$ (when $l$ is even). Since the wheel $W_l$ has a triangle, and with $l$ even is $3$-colorable, we have the following corollary of Corollary \ref{c:homomorphism} and Theorems~\ref{t:clique} and \ref{c:lb}. 

\begin{prop}
Fix integers $r \ge 2, k \ge 1$. Then, $h(n,k,K_3) \leq h(n,k,W_{2r})\leq h(n,k,K_3)+o(n^2)$.  In particular, for $n$ sufficiently large, we have $h(n,k,W_{2r})=\Theta(n^2/k^2)$.
\end{prop}

We suspect $h(n, k, W_l)$ depends significantly on the parity of $l$, and this is related to odd wheels not being $3$-colorable. The following proposition makes an initial observation for odd wheels.   

\begin{prop}
There are positive constants $c_1,c_2>0$ such that the following holds. For fixed positive integers $r,k$ and every sufficiently large positive integer $n$, we have
$$c_1 \frac{n^2}{k^{2}} \le h(n, k, W_{2r+1}) \le c_2 \frac{n^2}{k^{3/2}}.$$
\end{prop}
\begin{proof}

As $K_3$ is a subgraph of $W_{2r+1}$, which in turn is a subgraph of $K_{2r+2}$, we have $h(n, k, W_{2r+1}) \geq h(n, k, K_3)  \geq c_1 \frac{n^2}{k^{2}}$ for a positive constant $c_1$, 
where the last bound is from Theorem \ref{c:lb} when $r=1$. For the upper bound, we observe that $W_{2r+1}$ is $4$-colorable as $C_{2r+1}$ is $3$-colorable, and hence $W_{2r+1}$ has a homomorphism to $K_4$. It follows from Theorems~\ref{t:hom} and \ref{t:clique} and $n$ is sufficiently larger that $h(n,k,W_{2r+1}) \leq h(n,k,K_4)+o(n^2) \leq c_2\frac{n^2}{k^{3/2}}$.  
\end{proof} 

Note that the lower and upper bounds are rather far apart, and the above result does not give an indication of which of the two bounds $h(n, k, W_{2r+1})$ is closer to. We obtain Theorem~\ref{t:wheel} by a careful analysis that combines the methods used to give upper bounds on $h(n, k, K_r)$ and $h(n, k, C_{2r+1})$. This gives a much better upper bound on $h(n, k, W_{2r+1})$ which we conjecture to be tight up to a constant factor depending only on $r$. 

\begin{proof}[Proof of Theorem \ref{t:wheel}]
Let $\ell$ be the largest integer at most $k$ which is twice a perfect $(r+1)$-power. We have $k \leq 2^{r+2}\ell$. 
Let $s=(\ell/2)^{1/(r+1)}$ and $t=\ell/(2s)$, so $s$ and $t$ are integers and $\ell=2st$. Let $H=W_{2r+1}$ and $v$ be the vertex of the wheel $W_{2r+1}$ of degree $2r + 1$, so $H_v = C_{2r+1}$. By Lemma~\ref{l:divide} and Theorem \ref{t:oddcycle}, we have 
 \begin{align*}
h(n, k, W_{2r+1}) &\le \frac{2n^2}{est^2} +  t \cdot h \left( \frac{2n}{t}, s, C_{2r+1} \right) \\
&\le \frac{2n^2}{est^2} +  t \cdot \beta_r \cdot \frac{(2n/t)^2}{s^{r + 1}} \\
&= \left(\frac{2}{e}+4\beta_r\right)
\frac{n^2}{(\ell/2)^{2-1/(r+1)}} \\
&\le 2^{2r+6}\left(\frac{2}{e}+4\beta_r\right)
\frac{n^2}{k^{2-1/(r+1)}},\\
\end{align*}
where $\beta_r$ only depends on $r$. Letting  $c_r=2^{2r+6}\left(\frac{2}{e}+4\beta_r\right)$ completes the proof.
\end{proof}

\section{From Max-$k$-Cuts to Max-$l$-Cuts}\label{s:lcuts}
Note that $h(G, k)$ and $\textsf{Max-$k$-Cut}(G)$ are related through the identity $$h(G, k) = e(G) - \textsf{Max-$k$-Cut}(G).$$ Consequently, the results proved in the previous section on $h(n, k, C_{2r+1})$ yield bounds on $\MC(G)$ for $C_{2r+1}$-free graphs $G$, and on $\textsf{Max-$l$-Cut}(G)$ for $l > 2$. 
We can also relate $\textsf{Max-$l$-Cut}(G)$ to $h(G, k)$ for $l < k$. 

\begin{defn}
For $G = (V, E)$, let $$d_l(G) = \frac{\textsf{Max-$l$-Cut}(G)}{|E|},$$
be the fraction of edges of $G$ that can cross an optimal $l$-cut of $G$.
\end{defn}

In particular, for a graph $G$ on $m$ edges, we have that $d_2(G) m = \MC(G)$. We can bound $d_l(G)$ in terms of $d_k(G)$ for $k \ge l$.

\begin{prop}\label{p:dk}
For $G = (V, E)$ and positive integers $l \le k$,
$d_l(G) \ge d_l(K_k) d_k(G). $
\end{prop}
\begin{proof}
For $G = (V, E)$ with $|V| = n, |E| = m$, fix a $k$-partition $V = V_1 \sqcup \cdots \sqcup V_k$ with $\textsf{Max-$k$-Cut}(G)$ edges between parts.
Choose a random, equitable partition of the set $\{1, \ldots , k\}$ into $l$ parts $S_1, \ldots ,S_l$ (so each part has size either $\lfloor k/l \rfloor$ or $\lceil k/l \rceil$). Let $W_i = \bigcup_{j \in S_i} V_j$ for $i = 1, \ldots, l$. Then $V = W_1 \sqcup \cdots \sqcup W_l$ is an $l$-partition of $V$.

We count the expected fraction of edges internal to the $l$-cut $V = W_1 \sqcup \cdots \sqcup W_l$. Any edge $e \in E$ internal to some $V_j$ will remain internal in the $l$-cut, and a fraction $1 - d_k(G)$ of the edges are of this form. All other edges $e \in E$ have endpoints in $V_{j_1}, V_{j_2}$ for $j_1 \neq j_2$. The probability that $e$ is internal in the $l$-cut is the probability that $V_{j_1}, V_{j_2} \in W_i$ for some single part of the $l$-partition, which is
$1 - d_l(K_k)$.
Thus, the expected fraction of edges internal to the $l$-cut $W_1, \ldots, W_l$ is 
\begin{equation*}\label{e:internaledge}
(1 - d_l(K_k))d_k(G) + (1 - d_k(G)) = 1 - d_l(K_k) d_k(G).
\end{equation*}
This gives the desired bound.
\end{proof}

By considering a uniformly random partition of a graph $G$, we have the bound 
$$\textsf{Max-$l$-Cut}(G) \geq \frac{l-1}{l} \cdot e(G).$$
\begin{defn} The \textit{surplus} of the \textsf{Max-$l$-Cut} of a graph $G$ is given by $$\pi(l,G):= \MLC(G)-\left(1-\frac{1}{l}\right)e(G).$$
The {\it surplus} of a graph $G$ is the \textit{surplus} of the \textsf{Max-$2$-Cut} of $G$.
\end{defn}
The surplus measures how much larger the \textsf{Max-$l$-Cut} is above the random bound. We are often interested in how large can we show the surplus is for graphs with certain properties.
We recall a standard method for giving a lower bound on the $\MLC(G)$ in terms of the $\MLC$ of smaller induced subgraphs of $G$.

\begin{lemma}\label{l:surplussubgraph}
Given $G = (V, E)$ and a vertex $k$-partition $V = V_1 \sqcup \cdots \sqcup V_k$, then 
$$\pi(l, G) \ge \sum_{i = 1}^k \pi(l, G[V_i])$$
\end{lemma}
\begin{proof}
For each $i$, fix an $l$-partition $V_i = W_{i1} \sqcup \cdots \sqcup W_{il}$ such that the number of edges between different $W_{ij}$ is $\MLC(G[V_i])$. From this, we construct an $l$-cut of $V = U_1 \sqcup \cdots \sqcup U_l$. For each $i$, fix a random permutation $\sigma \in S_l$ and assign $W_{ij}$ to $U_{\sigma(j)}$. Note that the $\MLC(G)$ is at least the expected size of this random $l$-cut. In this process, all edges not contained in $G[V_i]$ for some $i$ have endpoints randomly assigned and thus cross the resulting cut with probability $1- 1/l$. Therefore, 
$$\MLC(G) \ge \left(1 - \frac1l \right) \left( e(G) - \sum_{i = 1}^k e(G[V_i]) \right) + \sum_{i = 1}^k \MLC(G[V_i]).$$
Since $\MLC(G[V_i]) = (1 - 1/l) e(G[V_i]) + \pi(l, G[V_i])$, the above inequality implies the desired result.
\end{proof}

\section{Max-Cut in graphs with a forbidden odd cycle}\label{s:maxcut}

We leverage the previous results to prove the following lower bound on Max-Cut for $C_{2r+1}$-free graphs. This bound will be helpful to apply to graphs which are reasonably dense. 

\begin{lemma} \label{l:maxl} For any positive integer $r$, there exists $c = c(r) > 0$ such that the following holds. If $G$ is a $C_{2r+1}$-free graph with $n$ vertices and $m$ edges, then
$$\textsf{Max-Cut}(G) \ge \frac{m}{2}+\Omega_r((m/n^2)^{1/r}m).$$
\end{lemma}
\begin{proof} Let $k$ be the smallest even integer that is at least $(2c_rn^2/m)^{1/r}$, where $c_r$ is the implicit constant in Theorem \ref{t:oddcycle}. In particular, $c_rn^2/k^{r+1} \leq m/(2k)$. Let $m_{0}$ be the number of edges we delete from $G$ to get a $k$-partite graph, so $d_{k}(G)m=m-m_0$.
By Theorem \ref{t:oddcycle}, we have $m_0\leq c_r n^2/k^{r+1}$. 
Note as $k$ is even, $$d_2(K_k) = \frac{\left(k/2\right)^{2}}{{k\choose 2}}=\frac{k}{2(k-1)}=\frac{1}{2}\left(1+\frac{1}{k-1}\right),$$
which is realized by assigning $k/2$ vertices to each of $2$ parts arbitrarily. By Proposition \ref{p:dk}, we obtain that
\begin{eqnarray*}\textsf{Max-Cut}(G) & = & d_{2}(G)m\geq d_{2}(K_{k})d_{k}(G)m=d_{2}(K_{k})d_{2}(K_{k})(m-m_0) \\ & \geq & \frac{1}{2}\left(1+\frac{1}{k-1}\right)\left(m-c_r\frac{n^2}{k^{r+1}}\right)
 \geq  \frac{1}{2}\left(1+\frac{1}{k-1}\right)\left(m-\frac{m}{2k}\right) = \frac{m}{2}+\frac{m}{4(k-1)}.\end{eqnarray*}   
This last inequality gives the desired bound. 
\end{proof}

We will be able to show another lower bound on $\MC(G)$ for graphs $G = (V, E)$, using the following result of Alon, Krivelevich, and Sudakov.
\begin{lemma}[Lemma 3.3~\cite{AL05}] \label{l:abspos}
There is an absolute positive constant $\epsilon$ such that for every positive constant $M$, there is a $\alpha=\alpha(M)>0$ with the following property. If $G=(V,E)$ is a graph with $m$ edges such that the induced subgraph on any set of $N\geq M$ vertices all of which have a common neighbor contains at most $\epsilon N^{3/2}$ edges, then $$\MC{(G)}\geq \frac{m}{2}+\alpha\sum_{v \in V}\sqrt{d(v)}.$$
\end{lemma}

In particular, this condition holds for graphs with a small forbidden cycle. 

\begin{cor}~\label{c:sqrtbd}
For every positive integer $k$ there is $\alpha = \alpha(k) > 0$ such that
if $G = (V, E)$ is a $C_{k}$-free graph, then $$\MC(G)\geq \frac{m}{2}+\alpha\sum_{v\in V}\sqrt{d(v)}.$$
\end{cor}
\begin{proof}
Let $M=M(k)=\epsilon^{-2}k^2$, where $\epsilon>0$ is the absolute constant from Lemma~\ref{l:abspos}. Let $G' = (V', E')$ be an induced subgraph of $G$ with $|V'| = N \ge M$ and $V' \subset N(v)$ for some $v \in V$. It suffices to show that $|E'| \le \epsilon N^{3/2}$. Since $G$ is $C_{k}$-free, $G'$ is $P_{k-1}$-free. Hence,  
$$|E'| \leq \text{ex}(N,P_{k-1}) < Nk/2 < \epsilon N^{3/2},$$
where the middle inequality is due to Erd\H{o}s and Gallai~\cite{ErGa59}.
\end{proof}

We leverage these results to give an improved upper bound on the $\MC$ of a $C_{2r+1}$-free graph on $m$ edges.

\begin{proof}[Proof of Theorem~\ref{max-cut-odd-cycle-free}]
Let $k=2r+1$ with $r$ a positive integer and observe that $(k+5)/(k+7)=1-1/(r+4)$. Let $G=(V,E)$ be a $C_{2r+1}$-free graph with $m$ edges. Let $U \subset V$ be the set of vertices of degree at least $m^{2/(r+4)}$. As the sum of the degrees of vertices in $G$ is $2m$, we have $|U| \leq (2m)/m^{2/(r+4)}=2m^{1-2/(r+4)}$.

If the induced subgraph $G[U]$ contains at least $m/2$ edges, then applying Lemma \ref{l:maxl} to $G[U]$, we obtain that the surplus of $G[U]$ is $\Omega_r((m/|U|^2)^{1/r}m)$, which is at least $\Omega_r(m^{1-1/(r+4)})$. Lemma~\ref{l:surplussubgraph} implies that the surplus of a graph as it least the surplus of any induced subgraph. Hence, the surplus of $G$ is also $\Omega_r(m^{1-1/(r+4)})$.

Otherwise, the induced subgraph $G[U]$ has less than $m/2$ edges, and so $\sum_{v \in V\setminus U} d(v) \geq m/2$. By Corollary \ref{c:sqrtbd}, the surplus of $G$ is, for some positive constant $c_r$, at least $$c_r\sum_{v \in V} \sqrt{d(v)} \geq c_r\sum_{v \in V \setminus U} \sqrt{d(v)} \geq c_r\frac{m/2}{m^{2/(r+4)}}\sqrt{m^{2/(r+4)}} =\Omega_r(m^{1-1/(r+4)}),$$
where the last inequality uses concavity of the function $f(x)=x^{1/2}$ and that $d(v) < m^{2/(r+4)}$ for all $v \in V \setminus U$. 
\end{proof}

\section{Concluding Remarks}\label{s:conclude}
Theorems~\ref{t:oddcycle} and \ref{c:lb} together determine  $h(n, k, C_{2r+1})$ up to a constant factor depending only on $r$ for sufficiently large $n$. 
The lower bound in Theorem~\ref{c:lb} relies on the construction of Alon of a relatively dense pseudorandom $C_{2r+1}$-free graph as in Theorem~\ref{t:girth}. A corresponding pseudorandom graph construction of $K_r$-free graphs does not exist apart from the case $r=3$ (as $K_3=C_3$). However, it is conjectured that such a graph exists. If so, the proof would carry over and we would obtain the following conjecture, which would show that Theorem~\ref{t:clique} is tight up to a constant function of $r$ for sufficiently large $n$.

\begin{conj}
If $r \geq 3$ and $n \gg k$, then $h(n, k, K_r) = \Omega_{r} \left(n^{2}/k^{(r-1)/(r-2)}\right)$. 
\end{conj}

It is known that other interesting results would follow from knowing the existence of such pseudorandom $K_r$-free graphs, including giving nearly tight bounds for off-diagonal Ramsey numbers (see \cite{MV21,HW20}).

It would also be interesting to get better bounds on the \textsf{Max-Cut} of $H$-free graphs for other forbidden subgraphs $H$. The methods of Sections~\ref{s:lcuts} and~\ref{s:maxcut} can be used to obtain improved bounds for some other $H$, like odd wheels.

\section*{Acknowledgements}
We'd like to thank Matthew Kwan for pointing out the reference~\cite{GUO20} from which Proposition~\ref{p:cyclefree} follows.

\end{document}